\documentclass[a4paper]{article}

\usepackage{amssymb}
\usepackage{mathrsfs}
\usepackage{amsthm}
\usepackage{amsmath}

\DeclareMathAlphabet{\mathpzc}{OT1}{pzc}{m}{it}

\newtheorem{Thm}{Theorem}[section]
\newtheorem{Cor}[Thm]{Corollary}
\newtheorem{Lem}[Thm]{Lemma}
\newtheorem{Conj}[Thm]{Conjecture}

\theoremstyle{definition}
\newtheorem{Rem}[Thm]{Remark}

\theoremstyle{definition}

\theoremstyle{definition}
\newtheorem{Def}[Thm]{Definition}

\theoremstyle{definition} 

\makeatletter
\@addtoreset{equation}{section}
\makeatother

\newcommand{\N}{\mathbb{N}}

\newcommand{\R}{\mathbb{R}}
\newcommand{\Z}{\mathbb{Z}}

\newcommand{\mb}{\mathbf}
\newcommand{\mbb}{\mathbb}
\newcommand{\mc}{\mathcal}
\newcommand{\mi}{\mathit}

\newcommand{\mr}{\mathrm}

\newcommand{\mscr}{\mathscr}

\newcommand{\GG}{\mb{G}}
\newcommand{\lra}{\longrightarrow}

\newcommand{\wrt}{with respect to }
\newcommand{\I}{\mc{I}}

\begin{document}

\title{Algebraicity of Nash sets and of their \\ asymmetric cobordism}

\author{Riccardo Ghiloni
and 
Alessandro Tancredi}

\date{}

\maketitle


\begin{abstract}
This paper deals with the existence of algebraic structures on compact Nash sets. We introduce the algebraic-topological notion of asymmetric Nash cobordism between compact Nash sets, and we prove that a compact Nash set is semialgebraically homeomorphic to a real algebraic set if and only if it is asymmetric Nash cobordant to a point or, equivalently, if it is strongly asymmetric Nash cobordant to a real algebraic set. As a conse\-quence, we obtain new large classes of compact Nash sets semialgebraically homeomorphic to real algebraic sets. To prove our results, we need to develop new algebraic-topological approximation procedures. We conje\-cture that every compact Nash set is asymmetric Nash cobordant to a point, and hence semialge\-braically homeomorphic to a real algebraic~set.

\ 

\begin{footnotesize}
\end{footnotesize}
\end{abstract}


\section{The algebraization problem: state of the art}

The problem of the exi\-stence of algebraic structures on topological spaces is an old and deep question of real algebraic geometry.

The nonsingular case has been completely settled. In 1952, Nash \cite{Na} showed that eve\-ry compact smooth manifold $M$ is diffeomorphic to a union of nonsingular connected components of a real algebraic set. The algebraization problem for $M$ arose: Is $M$ diffeomorphic to a whole nonsingular real algebraic set? Five years later, Wallace \cite{wa} proved that the answer is affirmative if $M$ is the boundary of a compact smooth manifold-with-boundary. This insight of using cobordism was of crucial importance. In fact, in 1965, Milnor \cite{Mi} showed that every compact smooth manifold is unoriented cobordant to a nonsingular real algebraic set and later, in 1973, Tognoli \cite{To} used this result to prove that the answer to the preceding question is always affirmative: every compact smooth mani\-fold has an algebraic model; namely, it is diffeomorphic to a nonsingular real algebraic set.

Cobordism can be employed also to make algebraic topological spaces having singularities. This was done by Akbulut and King in their theory of reso\-lution towers (see \cite{AK}). Their idea is to consider compact topological spaces $M$ whose singularities can be topologically resolved in the following way: there exist finite families of compact smooth manifolds $\{M_i\}_i$ and of smooth maps $\{f_{ij}:M_{ij} \lra M_j\}_{i,j}$ from subsets $M_{ij}$ of $M_i$ to $M_j$ such that the quotient space obtained by gluing the $M_i$'s along the $f_{ij}$'s is homeo\-morphic to $M$. The $M_{ij}$'s are finite unions of smooth hypersurfaces of $M_i$ in general position. Now, one can topologically identify $M$ with the pair $\mc{I}=(\{M_i\}_i,\{f_{ij}\}_{i,j})$, which is said to be a resolution tower for $M$. Observe that Hironaka's desingularization theo\-rem \cite{Hi} ensures the existence of a resolution tower for every compact real algebraic set. The resolution tower $\mc{I}$ for $M$ is said to be a ``weak boundary'' if there exists another resolution tower $(\{M'_i\}_i,\{f'_{ij}:M'_{ij} \lra M'_j\}_{i,j})$ such that, for every $i,j$, $M'_i$ is a compact smooth manifold-with-boundary, $\partial M'_i=M_i$, $M'_{ij}$ is a finite union of smooth hypersurfaces of $M'_i$ in general position, $M_{ij}=M'_{ij} \cap M_i$ and $f_{ij}=f'_{ij}|_{M_{ij}}$. Thanks to this notion, Akbulut and King obtained a complete topological characterization of real algebraic sets with isolated singularities (see \cite{akis}) and, with the help of L. Taylor, they proved the existence of algebraic structures on every compact PL manifolds (see \cite{AK81,AT}). Furthermore, they found local topological necessary conditions for a compact polyhedron $P$ of dimension $\leq 3$ to be homeomorphic to a real algebraic set. In dimension $\leq 2$, these conditions reduce to the Sullivan condition: the link of each vertex of $P$ has even Euler chara\-cteristic. In dimension $3$, the mentioned conditions define five independent local topological obstru\-ctions for the algebraicity of $P$, including the Sullivan one. By employing the concept of ``weak boundary'', Akbulut and King proved that these local topological obstructions are the unique obstru\-ctions for a compact polyhedron of dimension $\leq 3$ to be homeomorphic to a real algebraic set. As a significant corollary, we have that every compact Nash set (or better every compact real analytic set) of dimension $\leq 3$ has a real algebraic structure. We refer the reader to \cite{BD,AK3} for the $2$-dimensional case and to \cite{AK} for the $3$-dimensional one. 

In \cite{McP1,McP2}, by means of a completely different method, McCrory and Paru\-si\'{n}ski di\-scovered local topological necessary conditions for the algebraicity of compact polyhedra in any dimension. These conditions coincide with the ones of Akbulut and King in dimension $\leq 3$. In higher dimension, the number of independent local topological obstructions defined by such necessary conditions is enormous, at least $2^{43}-43$.

The latter result has an important consequence: in arbitrary dimension, there is no hope to give a reasonable local topological description of compact polyhedra admitting a real algebraic structure. However, if compact polyhedra we want make algebraic have a Nash structure, then they satisfy all the McCrory-Parusi\'{n}ski conditions (see \cite{No,gh07}). Furthermore, such polyhedra admit resolution of singulari\-ties via Hironaka's desingularization theorem, and hence the method of resolution towers applies. In this way, at the present time, the following seems to be the most promising and meaningful formulation of the algebraization problem:

\vspace{1em}

\noindent \textsl{Algebraization problem:} \textit{Is a Nash set semialgebraically homeo\-morphic to a real algebraic set?}

\vspace{1em}

A similar question can be restated in the real analytic setting.

Here we are interested in the compact Nash case only.

It is important to remark that the theory of resolution towers furnishes a na\-tural and intriguing ``three-steps'' strategy to tackle the preceding problem for an arbitrary compact Nash set $M$: (I) resolve the singularities of $M$ obtaining a resolution tower $\mc{I}$; (II) make algebraic $\I$ via algebraic approximation techni\-ques, obtaining an algebraic resolution tower~$\mc{I}'$; (III) blow down $\mc{I}'$ obtaining a real algebraic set semialgebraically homeomorphic to~$M$.

As we have just recalled, the existence of $\I$ in step I is ensured by Hirona\-ka's desingularization theorem. The first part of step III can be performed as well: one can show that the blowing down of every algebraic resolution tower is semialgebraically homeomorphic to a real algebraic set. A serious difficulty appears in step~II. In fact, the blowing down of a (generic) resolution tower is topologi\-cally unstable with respect to the usual algebraic approximation techniques. Hence one cannot conclude that the blowing down of $\I'$ is semialgebraically homeomorphic to $M$. For this reason, one needs to require that the reso\-lution tower $\I$ of step I is of a very special type. Unfortunately, as one reads at page 173 of the 1992 book \cite{AK} (for the algebraic case), in order to obtain such a type of resolution tower, it seems to be necessary an ``as yet unproven reso\-lution of singularities theorem'' for Nash maps. This assertion is still true nowadays. The preceding consi\-derations fully describe the intrinsic difficulty to treat the \textsl{algebraization problem} by the method of resolution towers.

Recently, by using different methods, some examples of compact Nash sets admitting real algebraic structures have been found in arbitrary dimension and with non-isolated singularities. The examples we refer to are the following two:
\begin{itemize}
 \item[$(\mr{Ex}1)$] The product of a standard sphere and of a compact Nash set is Nash isomorphic to a real algebraic set (see \cite{Ta-To}).
 \item[$(\mr{Ex}2)$] Let $M \subset \R^n$ be a compact Nash set symmetric \wrt a point $p \in \R^n$. Then $M$ is Nash isomorphic to a real algebraic set if $p \not\in M$ or if $p$ is an isolated point of $M$, and it is semialgebraically homeomorphic to a real algebraic set if $p$ is a non-isolated point of $M$. Under suitably conditions, these results extend to the case in which $M$ is symmetric \wrt an affine subspace of $\R^n$ of positive dimension (see \cite{GT2,GT1}).
\end{itemize}

To the best of our knowledge, in the singular setting, until now, the \textsl{algebraization problem} has been solved only for the compact Nash sets described above; that is, the compact Nash sets of dimension $\leq 3$ or with isolated singularities, or the ones mentioned in the preceding two examples. 

On the contrary, the study of the algebraization problem for manifolds has been deepened in several directions. In \cite{Ba,BK}, it is proved that, for every compact smooth manifold $M$ of positive dimension, the set of birationally noniso\-morphic algebraic models of $M$ has the power of continuum. In \cite{BG}, Ballico~and the first author improved this result by showing that the algebraic structure~of every nonsingular real algebraic set of positive dimension can be deformed by an arbitrarily large number of effective parameters (see \cite{BG2,Gh} for the~singular~case).

The existence of several distinct algebraic models of $M$ poses the question of constructing algebraic models of $M$ with additional algebraic-geometric pro\-perties or algebraic models of $M$ on which certain smooth objects attached to $M$ become algebraic. There is a wide literature devoted to this topic. For some of the main developments of the algebraization problem for manifolds, we refer the reader to two recent papers, and to the numerous refe\-rences mentioned therein: see \cite{k11}, especially the first section, for the  algebraization of manifold pairs, of vector bundles, of homology and cohomology classes, of smooth submanifolds of euclidean spaces via ambient isotopies, and of analytic hypersurfaces with isolated singularities via ambient isotopies; see \cite{ma14} for the Nash rationality conjecture concerning the existence of rational algebraic models. 


\section{The results} \label{sec:results}

The goal of this paper is to introduce the new algebraic-topological notion of asymmetric Nash cobordant compact Nash sets, and to use it to deal with the \textsl{algebraization problem}. We prove that a compact Nash set is semialgebraically homeomorphic to a real algebraic set if and only if it is (strongly) asymmetric Nash cobordant to a compact real algebraic set or, equivalently, if it is asymmetric Nash cobordant to a point. As an application, we obtain quite gene\-ral algebraization theorems, which describe, by means of transversa\-lity, new large classes of compact Nash sets semialgebraically homeomorphic to real algebraic sets.

In what follows, we use standard notions and results concerning real algebraic, Nash and semialgebraic sets. As usual, we assume that these sets are equipped with the euclidean topology. Our standard reference is \cite{BCR} (see also \cite{BR,Sh}).

Let us precise only the meaning we give to some basic notions from Nash geometry. Let $M$ be a locally closed semialgebraic subset of $\R^m$. Given an open semialgebraic subset $\Omega$ of $\R^m$ containing $M$, we say that $M$ is a Nash subset of $\Omega$ if it is the zero set of a finite family of Nash functions defined on $\Omega$. We underline that, by the results of \cite{CS3}, $M$ is a Nash subset of $\R^m$ if and only if $M$ is closed in $\R^m$ and it is a Nash subset of one of its open semialgebraic neighborhoods (see \cite{TT04}). By a Nash set, we mean a Nash subset of an open semialgebraic subset of some $\R^m$. A Nash manifold is a nonsingular Nash set; namely, a Nash submanifold of some $\R^m$ (see \cite[Definition 2.9.9]{BCR}). Let $M \subset \R^m$ and $N \subset \R^n$ be two Nash sets. A map $f:M \lra N$ is a Nash map if there exist an open semialgebraic neighborhood $U$ of $M$ in $\R^m$ and an extension $F:U \lra \R^n$ of $f$ from $U$ to $\R^n$, which is Nash; namely, semialgebraic and of class $\mscr{C}^\infty$.

We recall that a semialgebraic Whitney stratification of a locally closed semialgebraic set $A \subset \R^m$ is a finite stratification, whose strata are Nash submani\-folds of $\R^m$ satisfying conditions $a$ and $b$ of Whitney. We refer the reader to \cite{GWdPL,Sh2} for the general properties of Whitney and semialgebraic Whitney strati\-fications. Given a Nash submanifold $B$ of $\R^m$, we say that \textit{$B$ is transverse to $A$ in $\R^m$} if there exists a semialgebraic Whitney stratification $\{A_i\}_i$ of an open semialgebraic neighborhood of $A \cap B$ in $A$ such that $B$ is transverse to each stratum $A_i$ in $\R^m$.

For short, we say that a Nash set has an \textit{algebraic structure} if it is semialgebraically homeomorphic to a real algebraic set. Given integers $n,m \in \N$ with $n<m$, we identify $\R^n$ with the vector subspace $\R^n \times \{0\}$ of $\R^m=\R^n \times \R^{m-n}$. If $S$ is a subset of $\R^m$, then we denote by $\mr{Zcl}_{\R^m}(S)$ the Zariski closure of $S$ in $\R^m$.

Let us introduce the mentioned new notion of asymmetric Nash cobordism.

\begin{Def} \label{def:nash-cobordant}
Let $M$ be a compact Nash set. Given a compact Nash subset $M'$ of $\R^n$, we say that $M$ is \emph{asymmetric Nash cobordant to $M'$}, or \emph{asym-Nash cobordant to $M'$} for short, if there exist a compact Nash subset $S$ of some $\R^{m+1}$ with $n \leq m$ and $M' \subset S$, and a compact Nash subset $N$ of $S$ with $N \cap M'=\emptyset$ such that:
\begin{itemize}
 \item[$(\mr{i})$] $N$ is semialgebraically homeomorphic to $M$.
 \item[$(\mr{ii})$] $\R^m \cap S=N \cup M'$.
 \item[$\mr{(iii)}$] $\R^m$ is transverse to $S$ in $\R^{m+1}$ locally at $N$. More precisely, there exists an open semialgebraic neighborhood $W$ of $N$ in $S$ such that $\R^m$ is transverse to $W$ in $\R^{m+1}$.
\end{itemize}

If, in addition, $S$ satisfies condition $(\mr{iv})$ below, then we say that $M$ is \emph{strongly asymmetric Nash cobordant to $M'$}, or simply \emph{strongly asym-Nash cobordant to $M'$}:
\begin{itemize}
 \item[$(\mr{iv})$] $S$ is a union of connected components of $\mr{Zcl}_{\R^{m+1}}(S)$.
\end{itemize}

The compact Nash set $M$ is called \emph{Nash boundary} if it is asym-Nash cobordant to the empty set.
\end{Def}

The use of adjective ``asymmetric'' is justified by the fact that the statement of the preceding definition is not symmetric in $M$ and $M'$. A deeper motivation is revealed by Theorem \ref{thm:equiv} below. The reader compares Definition \ref{def:nash-cobordant} with the classical ``symmetric'' notion of cobordism between Whitney stratified compact subsets of $\R^m$ (see \cite[Definition 3.3]{go} with $X=\R^m$).

The concept of asymmetric Nash cobordism just defined generalizes the standard one of unoriented cobordism between compact smooth manifolds. Indeed, if $M$ and $M'$ are two unoriented cobordant compact Nash manifolds and $T$ is a compact smooth manifold-with-boundary having their disjoint union $M \sqcup M'$ as boundary, then, by standard Nash approximation results (see \cite{BCR,Sh}), the double of $T$ can be embedded into some $\R^{m+1}$ as a compact Nash submanifold $S$ in such a way that $\R^m \cap S$ is Nash isomorphic to $M \sqcup M'$ and $\R^m$ is transverse to $S$ in $\R^{m+1}$. In particular, it follows that, if a compact Nash manifold is a smooth boundary; namely, it is the boundary of a compact smooth manifold-with-boundary, then it is also a Nash boundary. It is worth noting that there exist topological obstructions for a compact Nash set to be a Nash boundary. Indeed, every Nash boundary is also a $\mscr{P}$-boundary (see \cite[Theorem 1.2]{gh07}). In particular, it must have even Euler characteristic.

The reader observes that, if $M$ is semialgebraically homeomorphic to $M'$, then it is also asym-Nash cobordant to $M'$. It suffices to set $S:=M' \times S^1 \subset \R^{n+2}=\R^n \times \R^2$, $N:=M' \times \{(1,0)\}$, $m:=n+1$ and to identify $M'$ with $M' \times \{(-1,0)\}$ in the preceding definition. For the same reason, if a compact Nash set has an algebraic structure, 
then it is also strongly asym-Nash cobordant to a compact real algebraic set.

Our main result proves the converse implication. It reads as follows.

\begin{Thm} \label{thm:main}
If a compact Nash set is strongly asym-Nash cobordant to a compact real algebraic set, then it has an algebraic structure.
\end{Thm}

At first glance, one may hope to prove Theorem \ref{thm:main} by adapting the classical proof of the Nash-Tognoli theorem via standard algebraic approximation results and semialgebraic Thom's first isotopy lemma. This is not the case.

Let us explain why. Suppose to have a compact Nash manifold $M$ we want to make algebraic. Then there exist a compact Nash submanifold $S$ of some $\R^{m+1}$ such that $\R^m$ is transverse to $S$ in $\R^{m+1}$, $\R^m \cap S$ is equal to the disjoint union $N \sqcup M'$ of a compact Nash manifold $N$ diffeomorphic (and hence Nash isomorphic) to $M$ and of a compact nonsingular real algebraic set $M'$, and $S$ is a union of connected components of $S':=\mr{Zcl}_{\R^{m+1}}(S)$. This is the first part of the mentioned classical proof: see, for example, Theorem 14.1.10 of \cite[p. 378]{BCR}, where the notations are very different from the ones used here.

The classical proof proceeds as follows. Consider the Nash function $g:S' \lra \R$ equal to the projection $(x_1,\ldots,x_{m+1}) \longmapsto x_{m+1}$ on $S$ and to $1$ on $S' \setminus S$. Observe that $0$ is a regular value of $g$ and $g^{-1}(0)=N \sqcup M'$. Since the real algebraic set $M'$ is nonsingular, it follows that it is ``quasi regular'' (see \cite{To1}, p. 51) or, that is the same, it is ``nice'' (see \cite[p. 57]{AK}). For this reason, we can apply the relative Weierstrass approximation theorem (see \cite[Teorema 1]{To} or \cite[Lemma 2.8.1]{AK}), obtaining a regular function $h:S' \lra\R$ arbitrarily $\mscr{C}^\infty$-close to $g$ and vanishing on $M'$. In this way, by smooth Thom's first isotopy lemma, there exists a small smooth isotopy $(F_t:g^{-1}(0) \lra S)_{t \in [0,1]}$ from $g^{-1}(0)$ to $h^{-1}(0)$ in $S$ fixing $M'$. In particular, $F_1(N) \sqcup M'$ is equal to the nonsingular real algebraic set $h^{-1}(0)$. Since $M'$ is Zariski closed in $\R^{m+1}$, it follows immediately that $F_1(N)$ is a nonsingular real algebraic set, as desired.

Suppose now that $M$ is an arbitrary (possibly singular) compact Nash set strongly asym-Nash cobordant to a compact real algebraic set $M'$. Let $S \subset \R^{m+1}$ and $N$ be as in Definition \ref{def:nash-cobordant}. The reader observes that there are two serious obstructions to adapt the preceding proof to the present singular situation. First, we do not require that $M'$ is nice; hence, it is not possible to apply the mentioned relative Weierstrass approximation theorem to the function $g$ (which can be defined as above). Secondly, we imposed the transversality between $\R^m$ and $S$ locally at $N$, but not locally at $M'$. In this way, also in the case in which one would have a regular map $h$ arbitrarily $\mscr{C}^\infty$-close to $g$ and vani\-shing on $M'$, semialgebraic Thom's first isotopy lemma would not ensure the existence of a semialgebraic isotopy from $g^{-1}(0)$ to $h^{-1}(0)$ in $S$ fixing $M'$.

Our strategy to prove Theorem \ref{thm:main} is based on a new algebraic-topological approximation procedure. First, we reduce to the case in which $M'$ is a point by using the real algebraic blowing down operation and then we perform the algebraic approximation of $g$ by means of an \textit{ad hoc} family of nonsingular hypersurfaces of $\R^{m+1}$ ``converging to the boundary of a cylinder''. More precisely, we obtain Theorem \ref{thm:main} as an immediate consequence of the following two results.

\begin{Thm} \label{thm:equiv}
Let $M$ be a compact Nash set. The following assertions are equivalent:
\begin{itemize}
 \item[$(\mr{i})$] $M$ is strongly asym-Nash cobordant to a compact real algebraic set.
 \item[$(\mr{ii})$] $M$ is strongly asym-Nash cobordant to a point.
 \item[$(\mr{iii})$] $M$ is asym-Nash cobordant to a point.
\end{itemize}
\end{Thm}

\begin{Thm} \label{thm:betlee}
If a compact Nash set is strongly asym-Nash cobordant to a point, then it has an algebraic structure.
\end{Thm}

We observe that, by definition, if a compact Nash set $M$ is asym-Nash cobordant to the empty set (namely, if $M$ is a Nash boundary), then it is also asym-Nash cobordant to a point. In this way, thanks to Theorem \ref{thm:equiv}, we have that $M$ is strongly asym-Nash cobordant to a point as well. By applying Theorem \ref{thm:betlee} to $M$, we infer at once the following significant result.

\begin{Cor} \label{cor:nash-boundaries}
Every Nash boundary has an algebraic structure.
\end{Cor}


Let us present some applications of our results.

First, we need to extend to Nash maps the notion of transversality between Nash submanifolds and locally closed semialgebraic subsets of $\R^m$, given above.

Let $A$ be a Nash set, let $N$ be a Nash manifold, let $B$ be a Nash subset of $N$ and let $f:A \lra N$ be a Nash map. We say that \textit{$f$ is transverse to $B$} if there exist a semialgebraic Whitney stratification $\{A_i\}_i$ of an open semialgebraic neighborhood of $f^{-1}(B)$ in $A$ and a semialgebraic Whitney stratification $\{B_j\}_j$ of an open semialgebraic neighborhood of $f(f^{-1}(B))$ in $N$ such that the restriction of $f$ to each stratum $A_i$ is transverse to each stratum $B_j$ in $N$. In the case in which $A$ is a Nash subset of $N$ and the inclusion map $A \hookrightarrow N$ is transverse to $B$, we say that \textit{$A$ is transverse to $B$ in $N$} and also that the set $A \cap B$ is the \textit{transverse intersection of $A$ and $B$ in $N$}.

In the latter situation, the condition of transversality between $A$ and $B$ in $N$ can be restated explicitly as follows: $A$ is transverse to $B$ in $N$ if there exist a semialgebraic Whitney stratification $\{A_i\}_i$ of an open semialgebraic neighborhood of $A \cap B$ in $A$ and a semialgebraic Whitney stratification $\{B_j\}_j$ of an open semialgebraic neighborhood of $A \cap B$ in $B$ such that each stratum $A_i$ is transverse to each stratum $B_j$ in $N$.

We remind the reader that a real algebraic set is said to have totally algebraic homology if each of its homology classes over $\Z_2$ can be represented by a real algebraic subset. Remarkable examples of compact nonsingular real algebraic sets with totally algebraic homology are the standard unit spheres and the grassmannians.

Thanks to Theorem \ref{thm:main}, we obtain the following quite general algebraization result.

\begin{Thm} \label{thm:a}
Let $X$ be a compact real algebraic set, let $Y$ be a nonsingular real algebraic set, let $Z$ be a real algebraic subset of $Y$ and let $f:X \lra Y$ be a Nash map transverse to $Z$. Then, in each of the following two cases, the compact Nash set $f^{-1}(Z)$ has an algebraic structure:
\begin{itemize}
 \item[$(\mr{i})$] $f$ is $\mscr{C}^0$-homotopic to a regular map.
 \item[$(\mr{ii})$] $X$ is a real algebraic subset of some compact nonsingular real algebraic set $V$ with totally algebraic homology, $Y$ has totally algebraic homology and the map $f$ admits a Nash extension from the whole $V$ to $Y$.
\end{itemize}
\end{Thm}

We remark that, in order to prove the latter theorem, we will need the semialgebraic version of a powerful isotopy result of Thom (see Corollary \ref{cor:thom}) and a new nonstandard version of a basic algebraic approximation result of Akbulut and King, the so-called workhorse theorem (see Lemma \ref{lem:approx} and Remark \ref{rem:nonstandard}). Furthermore, the asymmetry of Defi\-nition \ref{def:nash-cobordant} will play a crucial role, at least in the proof of point $(\mr{i})$ (see Remarks \ref{rem:2.6i}, \ref{rem:2.6ii} and \ref{rem:2.6ii-bis}). Finally, as a bypro\-duct of the proof of point $(\mr{ii})$, we obtain a new algebraic approximation result for Nash maps between nonsingular real algebraic sets with totally algebraic homology (see Theorem \ref{thm:approx-new}), which is interesting in its own right. 

Our next result gives two simple ways of constructing Nash boundaries, which have algebraic structures, as we saw in Corollary \ref{cor:nash-boundaries}. As usual, a subset of $\R^m$ is said to be proper if it is different from $\R^m$.

\begin{Thm} \label{thm:b}
In each of the following two cases, the compact Nash set $M$ is a Nash boundary and hence it has an algebraic structure:
\begin{itemize}
 \item[$(\mr{i})$] $M$ is the product of a Nash boundary (for example, a compact Nash mani\-fold which is a smooth boundary) by an arbitrary compact Nash set.
 \item[$(\mr{ii})$] $M$ is the transverse intersection of two proper Nash subsets of some $\R^m$, one of which is compact.
\end{itemize}
\end{Thm}

In the semialgebraic setting, preceding point $(\mr{i})$ generalizes $(\mr{Ex}1)$.

We have a conjecture.

\begin{Conj}
Every compact Nash set is strongly asym-Nash cobordant to a compact real algebraic set or, equivalently, every compact Nash set is asym-Nash cobordant to a point.
\end{Conj}

The reader observes that the validity of this conjecture, combined with Theo\-rem \ref{thm:main}, would imply the complete affirmative solution of the \textsl{algebraization problem} in the compact case: ``\textsl{every compact Nash set has an algebraic stru\-cture}''.

The proofs of our results are given in the next section.


\section{Proofs of the results}

In this section, we give the proofs of our results. We divide the section into four sub\-sections. In the first, we prove Theorem \ref{thm:equiv}. The second subsection is devoted to the proofs of Theorems \ref{thm:betlee} and \ref{thm:main}. The last two subsections deal with Theorems \ref{thm:a} and \ref{thm:b}.


\subsection{Proof of Theorem \ref{thm:equiv}}

Implications $(\mr{ii}) \Longrightarrow (\mr{i})$ and $(\mr{ii}) \Longrightarrow (\mr{iii})$ are evident. Implication $(\mr{iii}) \Longrightarrow (\mr{ii})$ is an easy consequence of the Artin-Mazur theorem.

\begin{proof}[Proof of implication $(\mr{iii}) \Longrightarrow (\mr{ii})$]
Suppose $M$ is asym-Nash cobordant to a point. By definition, there exist a compact Nash subset $S$ of some $\R^{m+1}$ such that $\R^m \cap S$ is equal to the disjoint union of a point $p$ and of a compact Nash set $N$ semialgebraically homeomorphic to $M$, and $\R^m$ is transverse to $S$ in $\R^{m+1}$ locally at $N$. Let $h:\R^{m+1} \lra \R$ be a Nash function having $S$ as its zero set. We apply the Artin-Mazur theorem to $h$ (see \cite[Theorem 8.4.4]{BCR}), obtaining a positive integer $k$, a nonsingular real algebraic subset $V$ of $\R^{m+2+k} = \R^{m+1} \times \R \times \R^k$, a semialgebraically connected component $V'$ of $V$ and a Nash isomorphism $\sigma:\R^{m+1} \lra V'$ such that $\pi(\sigma(x))=x$ and $\tau(\sigma(x))=h(x)$ for every $x \in \R^{m+1}$, where $\pi:\R^{m+1} \times \R \times \R^k \lra \R^{m+1}$ and $\tau:\R^{m+1} \times \R \times \R^k \lra \R$ denote the natural projections. Let $H$ be the coordinate hyperplane $x_{m+1}=0$ of $\R^{m+2+k}$, let $S':=\mr{Zcl}_{\R^{m+2+k}}(\sigma(S))$ and let $g:V \lra \R$ be the restriction of $\tau$ to $V$. Observe that $H \cap \sigma(S)=\sigma(N) \sqcup \{\sigma(p)\}$ and $H$ is transverse to $\sigma(S)$ in $\R^{m+2+k}$ locally at $\sigma(N)$. Moreover, since $\sigma(S)=V' \cap g^{-1}(0)$ is a union of connected components of $g^{-1}(0)$ and $S' \subset g^{-1}(0)$, we infer at once that $\sigma(S)$ is a union of connected components of $S'$ as well. This proves that $M$ is strongly asym-Nash cobordant to the point $\sigma(p)$.
\end{proof}

In order to prove implication $(\mr{i}) \Longrightarrow (\mr{ii})$, our idea is to use a suitable version of the real algebraic blowing down lemma by Akbulut and King, which allows to make algebraic the topological operation of ``collapsing a subspace to a point''. In this context, a key notion is the one of projectively closed real algebraic set. Let $j_n:\R^n \lra \mbb{P}^n(\R)$ be the affine chart sending $x$ into $[1,x]$. A real algebraic set $X \subset \R^n$ is called projectively closed if $j_n(X)$ is Zariski closed in $\mbb{P}^n(\R)$. This condition can be restated in terms of overt polynomials. A polynomial $P \in \R[X_1,\ldots,X_n]$ is overt if its homogeneous leading term (namely, its homogeneous part of maximum degree) vanishes only at the origin of $\R^n$. It is easy to verify that the real algebraic subset $X$ of $\R^n$ is projectively closed if and only if there exists an overt polynomial in $\R[X_1,\ldots,X_n]$ whose zero set coincides with $X$. The reader is referred to  Sections II.3 and II.6 of \cite{AK} for more details on this topic.

\begin{Lem} \label{lem:projectively-closed}
If a compact Nash set is strongly asym-Nash cobordant to a compact real algebraic set, then it is also strongly asym-Nash cobordant to a projectively closed real algebraic set
\end{Lem}
\begin{proof}
Let $M$ be a compact Nash set strongly asym-Nash cobordant to a compact real algebraic set. By definition, there exist a compact Nash subset $S$ of some $\R^{m+1}$ such that $\R^m \cap S$ is the disjoint union of a compact real algebraic set $M'$ and of a compact Nash set $N$ semialgebraically homeomorphic to $M$, $\R^m$ is transverse to $S$ in $\R^{m+1}$ locally at $N$ and $S$ is the union of certain connected components of $\mr{Zcl}_{\R^{m+1}}(S)$. Thanks to Theorem 2.5.13 of \cite{AK} (or, better, thanks to its proof with $W:=\R^m$ and $V:=M'$), we know that there exists a biregular embedding $\psi:\R^m \lra \R^k$ (namely, $\psi(\R^m)$ is Zariski closed in $\R^k$ and the restriction of $\psi$ from $\R^m$ to $\psi(\R^m)$ is a biregular isomorphism) such that $\psi(M')$ is projectively closed. Denote by $\Psi:\R^{m+1} \lra \R^{k+1}$ the biregular embedding $\psi \times \mi{id}_{\R}$, where $\mi{id}_{\R}:\R \lra \R$ is the identity map on $\R$. The reader observes that $\R^k \cap \Psi(S)=\psi(N) \sqcup \psi(M')$, $\R^k$ is transverse to $\Psi(S)$ in $\R^{k+1}$ locally at $\psi(N)$ and $\Psi(S)$ is a union of connected components of $\mr{Zcl}_{\R^{k+1}}(\Psi(S))=\Psi(\mr{Zcl}_{\R^{m+1}}(S))$. It follows that $M$ is strongly asym-Nash cobordant to the projectively closed real algebraic set $\psi(M')$.
\end{proof}

We are ready to complete the proof of Theorem \ref{thm:equiv}.

\begin{proof}[Proof of implication $(\mr{i}) \Longrightarrow (\mr{ii})$] By Lemma \ref{lem:projectively-closed}, we can assume that $M$ is stron\-gly asym-Nash cobordant to a proje\-ctively closed real algebraic set. Then there exist a compact Nash subset $S$ of some $\R^{m+1}$ such that $\R^m \cap S$ is equal to the disjoint union of a projectively closed real algebraic set $M' \subset \R^m$ and of a compact Nash set $N$ semialgebraically homeomorphic to $M$, $\R^m$ is transverse to $S$ in $\R^{m+1}$ locally at $N$ and $S$ is a union of connected components of $S':=\mr{Zcl}_{\R^{m+1}}(S)$.

Let $D \in \R[X]=\R[X_1,\ldots,X_{m+1}]$ be a (non-zero) polynomial whose zero set is $S'$ and let $d$ be its degree. Since $M'$ is projectively closed as a real algebraic subset of $\R^m$, it is also projectively closed as a real algebraic subset of $\R^{m+1}$: indeed, $j_m(M')=j_{m+1}(M')$ and $\mbb{P}^m(\R)$ is a Zariski closed subset of $\mbb{P}^{m+1}(\R)$, up to natural identifications. In this way, there exists an overt polynomial $E \in \R[X]$ having $M'$ as its zero set. If $M'=\emptyset$, then we define $E(X):=1+\sum_{i=1}^{m+1}X_i^2$. Write $E$ as follows: $E=\sum_{j=0}^eE_j$, where $e:=\deg(E)$ and $E_j$ is a homogeneous polynomial in $\R[X]$ of degree $j$. Since $E$ is overt and not constant, we have that $e \geq 2$ and $E_e$ vanishes only at $0$. Let $\ell$ be a positive integer such that $e\ell>d$.

Denote by $x=(x_1,\ldots,x_{m+1})$ the coordinates of $\R^{m+1}$. Define the polynomial $L \in \R[X,X_{m+2}]$, the polynomial maps $\alpha:\R^{m+1} \lra \R^{m+2}$ and $\beta:\R^{m+2} \lra \R^{m+2}$, and the subsets $S_*$, $S'_*$ and $N_*$ by setting
\begin{align*}
&L(X,X_{m+2}):=D(X)^2+(X_{m+2}-E(X))^{2\ell},\\
&\alpha(x):=(x,E(x)),\\
&\beta(x,x_{m+2}):=(xx_{m+2},x_{m+2}),\\
&S_*:=\{0\} \cup \beta(\alpha(S)),\\
&S'_*:=\{0\} \cup \beta(\alpha(S')),\\
&N_*:=\beta(\alpha(N)).
\end{align*}

Observe that: $\alpha(S') \cap \R^{m+1}=M' \times \{0\}$, the restriction of $\beta \circ \alpha$ from $\R^{m+1} \setminus M'$ to $\R^{m+2} \setminus \R^{m+1}$ is a biregular embedding, $S_*$ is a compact semialgebraic set equal to the union of certain connected components of $S'_*$, $N_*$ is a compact Nash set semialgebraically homeomorphic to $M$ and $L$ is a polynomial of degree $2e\ell$ with $\alpha(S')$ as its zero set and with $E_e^{2\ell}$ as its homogeneous leading term. Moreover, it holds:
\[
\beta(\alpha(S')) \setminus \R^{m+1}=\big\{(x,x_{m+2}) \in \R^{m+2} \setminus \R^{m+1} \, \big| \, L(\beta^{-1}(x,x_{m+2}))=0\big\}.
\]
By clearing denominators in the rational expression of $L(X/X_{m+2},X_{m+2})$, we infer at once the existence of a polynomial $R \in \R[X,X_{m+2}]$ such that
\[
L(\beta^{-1}(x,x_{m+2}))=x_{m+2}^{-2e\ell}\big(E_e(x)^{2\ell}+x_{m+2}R(x,x_{m+2})\big)
\]
for every $(x,x_{m+2}) \in \R^{m+2} \setminus \R^{m+1}$. Since $E_e$ vanishes only at $0$, we infer that $S'_*$ is equal to the zero set of the polynomial $E_e(X)^{2\ell}+X_{m+2}R(X,X_{m+2})$ and hence it is Zariski closed in $\R^{m+2}$. It follows that $S_*$ is a union of connected components of $\mr{Zcl}_{\R^{m+2}}(S_*)$.

Denote by $H$ the coordinate hyperplane $x_{m+1}=0$ of $\R^{m+2}$. It is immediate to verify that $H \cap S_*=N_* \sqcup \{0\}$ and, since $N_* \cap \R^{m+1}=\emptyset$, $H$ is transverse to $S_*$ in $\R^{m+2}$ locally at $N_*$.

Summarizing, we have proved that: $S_*$ is a compact Nash subset of $\R^{m+2}$, $H \cap S_*$ is the disjoint union of $\{0\}$ and of the compact Nash set $N_*$ semialgebraically homeomorphic to $M$, $H$ is transverse to $S_*$ in $\R^{m+2}$ locally at $N_*$ and $S_*$ is a union of connected components of its Zariski closure in $\R^{m+2}$. This ensures that $M$ is strongly asym-Nash cobordant to a point, as desired.
\end{proof}


\subsection{Proofs of Theorems \ref{thm:betlee} and \ref{thm:main}}

We begin with three preliminary lemmas and a corollary.

The first result is a simple consequence of the semialgebraic version of Thom's first isotopy lemma of Coste and Shiota (see \cite{CS}).

\begin{Lem}\label{lem:thom}
Let $A$ be a locally closed semialgebraic subset of $\R^m$ equipped with a semialgebraic Whitney
stratification $\{A_i\}_i$, and let $u:U \lra \R^k$ be a Nash map from an open semialgebraic neighborhood $U$ of $A$ in $\R^m$ to some $\R^k$ such that $0$ is a regular value of the restriction of $u$ to each stratum $A_i$ and $u^{-1}(0) \cap A$ is compact. Choose a compact neighborhood $C$ of $u^{-1}(0) \cap A$ in $U$.

Suppose to have a Nash map $v:U \lra \R^k$ with the following property:
\begin{itemize}
 \item[$(\ast)$] there exists a positive real number $\varepsilon$ such that the set $L=\{(x,s) \in A \times [-\varepsilon,1+\varepsilon] \, | \, (1-s)u(x)+sv(x)=0\}$ is compact and $\pi_A(L) \subset C$, where $\pi_A: A \times \R \lra A$ denotes the natural projection onto $A$.
\end{itemize}

Then, if the Nash map $v$ is also sufficiently $\mscr{C}^1$-close to $u$ on $C$, there exists a semialgebraic homemorphism
$\theta:u^{-1}(0) \cap A \lra v^{-1}(0) \cap A$ such that, for every $i$, $\theta(u^{-1}(0) \cap A_i)=v^{-1}(0) \cap A_i$ and the restriction of $\theta$ from $u^{-1}(0) \cap A_i$ to $v^{-1}(0) \cap A_i$ is a Nash isomorphism.
\end{Lem}
\begin{proof} Let $\rho:\R \lra \R$ be a Nash embedding and let $a,b \in \R$ such that $\rho(\R)=(-\varepsilon,1+\varepsilon)$, $\rho(a)=0$ and $\rho(b)=1$. For example, one can set
\[
\rho(t):=(1+(1+2\varepsilon)t(1+t^2)^{-1/2})/2.
\]
Consider the Nash map $\varphi:U \times \R \lra \R^k$ defined by setting
\[
\varphi(x,t):=(1-\rho(t))u(x)+\rho(t)v(x).
\]
Define the closed semialgebraic subset $B$ of $U \times \R$, the partition $\{B_i\}_i$ of $B$ and the Nash map $\pi:U \times \R \lra \R$ as follows:
\begin{align*}
&B:=\{(x,t) \in A \times \R \, | \, \varphi(x,t)=0\},\\
&B_i:=\{(x,t) \in A_i \times \R \, | \, \varphi(x,t)=0\} \; \text{ for every $i$},\\
&\pi(x,t):=t.
\end{align*}

By condition $(\ast)$, it follows immediately that the restriction of $\pi$ to $B$ is a proper map. Furthermore, if $v$ is sufficiently $\mscr{C}^1$-close to $u$ on $C$, then $\{B_i\}_i$ is a semialgebraic Whitney stratification of $B$ and the restriction of $\pi$ to each stratum $B_i$ is a Nash submersion. Thanks to the semialgebraic version of Thom's first isotopy lemma \cite{CS}, there exists
a semialgebraic homeomorphism from $\pi^{-1}(a) \cap B=(u^{-1}(0) \cap A) \times \{a\}$ to $\pi^{-1}(b) \cap B=(v^{-1}(0) \cap A) \times \{b\}$, which induces a Nash isomorphism from $\pi^{-1}(a) \cap B_i=(u^{-1}(0) \cap A_i) \times \{a\}$ to $\pi^{-1}(b) \cap B_i=(v^{-1}(0) \cap A_i) \times \{b\}$ for each $i$. This completes the proof.
\end{proof}

The reader observes that, if $A$ is compact, then condition $(\ast)$ is always satisfied (with an arbitrary compact neighborhood $C$ of $A$ in $U$ and with an arbitrary $\varepsilon>0$).

Let us present a useful corollary of the preceding lemma. It is a semialgebraic version of a classical isotopy result of Thom (see \cite[Th\'eor\`eme 2.D.2, pp. 270-271]{Th}). Certainly, such a semialgebraic version is known to the experts in semialgebraic geometry. However, to the best of our knowledge, it is explicitly stated and proved here for the first time.

\begin{Cor} \label{cor:thom}
Let $X \subset \R^n$ be a compact semialgebraic set equipped with a semialgebraic Whitney stratification $\{X_i\}_i$, let $Y$ be a Nash manifold, let $Z$ be a closed semialgebraic subset of $Y$ equipped with a semialgebraic Whitney strati\-fication $\{Z_j\}_j$ and let $f:\Omega \lra Y$ be a Nash map from an open semialgebraic neighborhood $\Omega$ of $X$ in $\R^n$ to $Y$ such that the restriction of $f$ to each stratum $X_i$ is transverse to each stratum $Z_j$ in $Y$. Choose a compact semialgebraic neighborhood $K$ of $X$ in $\Omega$.

Then, if $g:\Omega \lra Y$ is a Nash map sufficiently $\mscr{C}^1$-close to $f$ on $K$, the semialgebraic sets $f^{-1}(Z) \cap X$ and $g^{-1}(Z) \cap X$ are semialgebraically homeomorphic.
\end{Cor}
\begin{proof}
Suppose $Y$ is a Nash submanifold of $\R^k$. Let $\rho:T \lra Y$ be a Nash tubular neighborhood of $Y$ in $\R^k$; namely, $T$ is an open semialgebraic neighborhood of $Y$ in $\R^k$, $\rho$ is a Nash retraction (for example, the closest point map) and $(T,\rho,Y)$ is a vector bundle (see \cite[Corollary 8.9.5]{BCR}). Define $Z^*:=\rho^{-1}(Z)$ and $Z^*_j:=\rho^{-1}(Z_j)$ for every $j$. Since $\rho$ is a Nash submersion, $\{Z^*_j\}_j$ turns out to be a semialgebraic Whitney stratification of the closed semialgebraic subset $Z^*$ of $T$. Let $i_Y:Y \hookrightarrow \R^k$ be the inclusion map and let $F,G:\Omega \lra \R^k$ be the Nash maps defined by $F:=i_Y \circ f$ and $G:=i_Y \circ g$. Consider the open semialgebraic subset $U:=\Omega \times T$ of $\R^{n+k}=\R^n \times \R^k$, the closed semialgebraic subset $A:=X \times Z^*$ of $U$, the semialgebraic Whitney stratification $\{X_i \times Z^*_j\}_{i,j}$ of $A$, and the Nash maps $u,v:U \lra \R^k$ defined as follows: $u(x,y):=F(x)-y$ and $v(x,y):=G(x)-y$ for every $(x,y) \in \Omega \times T=U$.

By hypothesis, the restriction of $f$ to each stratum $X_i$ of $X$ is transverse to each stratum $Z_j$ of $Z$ in $Y$. This is equivalent to assert that the restriction of $F$ to each stratum $X_i$ of $X$ is transverse to each stratum $Z^*_j$ of $Z^*$ in $\R^k$. The latter transversality condition is in turn equivalent to the fact that $0$ is a regular value of the restriction of $u$ to each stratum $X_i \times Z^*_j$ of $A$.

Let $T^*$ be a compact neighborhood of $Y$ in $T$, let $Z^{**}$ be the compact set $Z^* \cap T^*$ and let $C$ be a compact neighborhood of $X \times Z^{**}$ in $K \times T$. Define the continuous map $H:X \times [-1,2] \lra \R^k$ by setting $H(x,s):=(1-s)F(x)+sG(x)$. Since $g$ is arbitrarily $\mscr{C}^1$-close to $f$ on $K$, we have that $G$ is arbitrarily $\mscr{C}^1$-close to $F$ on $K$ as well. In particular, $G$ is also $\mscr{C}^0$-close to $F$ on $X$. In this way, we can assume that $H(X \times [-1,2])$ is a subset of $T^*$. Define the set $L$ as follows:
\[
L:=\big\{(x,y,s) \in A \times [-1,2] \, \big| \, (1-s)u(x,y)+sv(x,y)=0\big\}.
\]
Observe that $(1-s)u(x,y)+sv(x,y)=(1-s)F(x)+sG(x)-y$ for every $(x,y,s) \in U \times [-1,2]$. It follows that $L$ is compact, because it is a closed subset of the compact set $X \times Z^{**} \times [-1,2]$.

We have just proved that $A$, $u$, $v$ and $C$ satisfy the hypothesis of Lemma \ref{lem:thom} (with $\varepsilon=1$). Thanks to this lemma, we infer that $u^{-1}(0) \cap A$ and $v^{-1}(0) \cap A$ are semialgebraically homeomorphic. It remains to show that $u^{-1}(0) \cap A$ and $v^{-1}(0) \cap A$ are semialgebraically homeomorphic to $f^{-1}(Z)$ and to $g^{-1}(Z)$, respectively. This is quite easy. Indeed, we have that
\[
u^{-1}(0) \cap A=\{(x,y) \in X \times Z \, | \, y=f(x)\}
\]
and hence $u^{-1}(0) \cap A$ is the graph of the restriction of $f$ to $f^{-1}(Z) \cap X$. Similarly, $v^{-1}(0) \cap A$ is the graph of the restriction of $g$ to $g^{-1}(Z) \cap X$.
\end{proof}

Fix a positive integer $m$. We denote by $\|x\|$ the euclidean norm $(\sum_{i=1}^mx_i^2)^{1/2}$ of the vector $x=(x_1,\ldots,x_m)$ of $\R^m$, by $B_m(r)$ the open ball $\{x \in \R^m \, | \, \|x\|<r\}$ of $\R^m$ centered at $0$ with radius $r$ and by $\bar{B}_m(r)$ the closure of $B_m(r)$ in $\R^m$.

\begin{Lem} \label{lem:Q}
Let $Q$ be a compact semialgebraic subset of $\R^{m+1}$ such that $Q \cap \R^m=\{0\}$. Then there exist a positive real number $\varepsilon$ and a positive integer $k$ such that
\begin{equation*}
|x_{m+1}|>\varepsilon \, \|x\|^{2k} \text{ for every $(x,x_{m+1}) \in Q \setminus \{0\}$}.
\end{equation*}
\end{Lem}
\begin{proof}
Let $\nu:(\R^m \setminus \{0\}) \times \R \lra (\R^m \setminus \{0\}) \times \R$ be the biregular automorphism sending $(x,x_{m+1})$ into $(x\|x\|^{-2},x_{m+1})$
and let $Q':=\nu(Q \setminus (\{0\} \times \R))$. Since $Q$ is compact in $\R^{m+1}$ and $Q \cap \R^m=\{0\}$, it is immediate to verify that $Q'$ is closed in $\R^{m+1}$ and $Q' \cap \R^m=\emptyset$. Consider the positive continuous semialgebraic function $f:Q' \lra \R$ defined by $f(x,x_{m+1}):=1/|x_{m+1}|$ and apply
Proposition 2.6.2 of \cite{BCR} to $f$. We obtain a positive real number $r$ and a positive integer $k$ such that
\[
1/|x_{m+1}|<r(1+\|x\|^2+x_{m+1}^2)^k \; \text{ for every $(x,x_{m+1}) \in Q'$}.
\]
Let $s$ be a positive real number so large that $Q \subset \bar{B}_m(s) \times [-s,s]$. Since $Q' \subset \R^m \times [-s,s]$, we infer that
\[
|x_{m+1}|>r^{-1}(1+\|x\|^2+x_{m+1}^2)^{-k} \geq r^{-1}(1+s^2+\|x\|^2)^{-k}
\]
for every $(x,x_{m+1}) \in Q'$. Define $\varepsilon:=r^{-1}((1+s^2)s^2+1)^{-k}>0$. If $(x,x_{m+1}) \in Q \setminus (\{0\} \times \R)$, then it holds:
\begin{align*}
|x_{m+1}| &>r^{-1}\left(1+s^2+\right\|x\|x\|^{-2}\left\|^2\right)^{-k}=\\
&=r^{-1}\left((1+s^2)\|x\|^2+1\right)^{-k}\|x\|^{2k} \geq \varepsilon \|x\|^{2k},
\end{align*}
as desired.
\end{proof}

For every positive real number $\delta$ and for every positive integer $k$, we define the open semialgebraic subset $\Omega_{\delta,k}$ of $\R^{m+1}$ by setting
\[
\Omega_{\delta,k}:=\big\{(x,x_{m+1}) \in \R^{m+1} \, \big| \, |x_{m+1}|<\delta\|x\|^{2k}\big\},
\]
and we denote its closure in $\R^{m+1}$ by $\overline{\Omega_{\delta,k}}$.

In the next lemma, we improve the geometric conditions in which a strongly asymmetric Nash cobordism to a point can be performed.

\begin{Lem} \label{lem:wallace-type}
Let $M$ be a compact Nash set, strongly asym-Nash cobordant to a point. Then there exist a compact Nash subset $S$ of some $\R^{m+1}$, a Nash subset $N$ of $S$ semialgebraically homeomorphic to $M$, a compact semialgebraic neighborhood $W$ of $N$ in $S$, a positive real number $\delta$ and a positive integer $k$ such that: $0 \not\in N$, $\R^m \cap S=N \cup \{0\}$, $\R^m$ is transverse to $W$ in $\R^{m+1}$, $0$ is an isolated point of $S \cap \overline{\Omega_{\delta,k}}$, $W=(S \cap \overline{\Omega_{\delta,k}} \, ) \setminus \{0\}$ and
\begin{align}
&S \subset B_m(1/2) \times (-1/2,1/2), \label{eq:s}\\
&\mr{Zcl}_{\R^{m+1}}(S) \setminus S \subset \R^{m+1} \setminus (\bar{B}_{m}(2) \times [-2,2]). \label{eq:zcl}
\end{align}
\end{Lem}
\begin{proof}
Since $M$ is strongly asym-Nash cobordant to a point, there exists a compact Nash subset $T$ of some $\R^{n+1}$, a compact Nash subset $Z$ of $T$ semialgebraically homeomorphic to $M$, an open semialgebraic neighborhood $U$ of $Z$ in $T$ equipped with a semialgebraic Whitney stratification $\{U_i\}_i$ and a point $q \in T \setminus Z$ such that $\R^n \cap T=Z \cup \{q\}$, $\R^n$ is transverse to each $U_i$ in $\R^{n+1}$ and $T$ is a union of connected components of $\mr{Zcl}_{\R^{n+1}}(T)$. Moreover, by using an affinity of $\R^{n+1}$ if necessary, we can assume that $q=0$ and $T \subset \bar{B}_{n+1}(1/4)$.

Let us move away $L:=\mr{Zcl}_{\R^{n+1}}(T) \setminus T$ from $T$ by using a classical method of A. H. Wallace (see \cite[Subsection 3.2]{wa}). Since $T$ and $L$ are disjoint closed subsets of $\R^{n+1}$, there exists a continuous function $f:\R^{n+1} \lra \R$ such that $f$ vanishes on $T$ and is constantly equal to $4$ on $L$. Apply the Weierstrass approximation theorem to the restriction of $f$ to $\bar{B}_{n+1}(3)$. We obtain a polynomial $P \in \R[X_1,\ldots,X_{n+1}]$ such that $|P| \leq 1/4$ on $T$, $P(0)=0$ and $P>3$ on $L \cap \bar{B}_{n+1}(3)$. Consider the biregular embedding $G:\R^{n+1} \lra \R^{n+2}$ sending $x \in \R^{n+1}$ into $(x,P(x)) \in \R^{n+1} \times \R=\R^{n+2}$. Define $S':=G(T)$, $N':=G(Z)$, $W':=G(U)$, $W'_i:=G(U_i)$ for every $i$, $m:=n+1$ and $H:=\{(x_1,\ldots,x_{m+1}) \in \R^{m+1} \, | \, x_m=0\}$. We have that $S'$ is a compact Nash subset of $\R^{m+1}$ contained in $\bar{B}_{m+1}(\sqrt{2}/4) \subset B_{m+1}(1/2)$, $H \cap S'$ is the disjoint union of $\{0\}$ and of the compact Nash set $N'$ semialgebraically homeomorphic to $M$, $\{W'_i\}_i$ is a semialgebraic Whitney stratification of the open semialgebraic neighborhood $W'$ of $N'$ in $S'$, $H$ is transverse to each $W'_i$ in $\R^{m+1}$ and
\[
\mr{Zcl}_{\R^{m+1}}(S') \setminus S'=G(L) \subset \R^{m+1} \setminus (\bar{B}_m(3) \times [-3,3]) \subset \R^{m+1} \setminus \bar{B}_{m+1}(3).
\]

Let $\Theta:\R^{m+1} \lra \R^{m+1}$ be the linear change of coordinates of $\R^{m+1}$, sending $(x_1,\ldots,x_{m-1},x_m,x_{m+1})$ into $(x_1,\ldots,x_{m-1},x_{m+1},x_m)$. Define $S:=\Theta(S')$, $N:=\Theta(N')$, $W^\bullet:=\Theta(W')$ and $W^\bullet_i:=\Theta(W'_i)$ for every $i$. The reader observes that $\Theta(H)=\R^m$,
\begin{align*}
& S \subset B_{m+1}(1/2)=\Theta(B_{m+1}(1/2)) \subset B_m(1/2) \times (-1/2,1/2),\\
& \bar{B}_m(2) \times [-2,2] \subset \bar{B}_{m+1}(2\sqrt{2}) \subset \bar{B}_{m+1}(3)=\Theta(\bar{B}_{m+1}(3))
\end{align*}
and hence
\[
\mr{Zcl}_{\R^{m+1}}(S) \setminus S \subset \R^{m+1} \setminus (\bar{B}_m(2) \times [-2,2]).
\]

In order to complete the proof, we must modify $W^\bullet$ (and its semialgebraic Whitney stratification $\{W^\bullet_i\}_i$) in such a way that it remains transverse to $\R^m$ in $\R^{m+1}$, but it assumes the required form. Restrict $W^\bullet$ around $N$ in such a way that
\begin{equation} \label{eq:3.5}
\text{$Q:=S \setminus W^\bullet$ is a neighborhood of $0$ in $S$.}
\end{equation}
Since $Q \cap \R^m=\{0\}$, by Lemma \ref{lem:Q}, there exists a positive real number $\varepsilon$ and a positive integer $k$ such that
\begin{equation} \label{eq:estimate}
\text{$|x_{m+1}|>\varepsilon\|x\|^{2k} \,$ for every $(x,x_{m+1}) \in Q \setminus \{0\}$.}
\end{equation}
For every $\delta \in (0,\varepsilon)$, define:
\begin{align*}
&W^*_{\delta}:=S \cap \overline{\Omega_{\delta,k}},\\
&W_{\delta}:=W^*_{\delta} \setminus \{0\},\\
&W_{\delta}^i:=W^\bullet_i \cap \Omega_{\delta,k} \; \text{ for every $i$},\\
&W^{i,\partial}_{\delta}:=W^\bullet_i \cap \partial\Omega_{\delta,k} \; \text{ for every $i$,}
\end{align*}
where $\partial\Omega_{\delta,k}$ denotes the boundary of $\Omega_{\delta,k}$ in $\R^{m+1}$. Thanks to \eqref{eq:estimate}, $W^*_{\delta}$ turns out to be a compact semialgebraic neighborhood of $N$ in $S$ contained in $W^\bullet \cup \{0\}$. By \eqref{eq:3.5}, $0$ is an isolated point of $W^*_{\delta}$. It follows at once that $W_{\delta}$ is a compact semialgebraic neighborhood of $N$ in $S$ contained in $W^\bullet$. Bearing in mind that $\R^m$ is transverse to each $W^\bullet_i$ in $\R^{m+1}$, if we choose $\delta$ sufficiently small, we have that $\partial\Omega_{\delta,k}$ is transverse to each $W^\bullet_i$ in $\R^{m+1}$ as well. It follows that the Nash submanifolds $\{W_{\delta}^i\}_i \cup \{W^{i,\partial}_{\delta}\}_i$ of $\R^{m+1}$ form a semialgebraic Whitney stratification of $W_{\delta}$, whose elements are transverse to $\R^m$ in $\R^{m+1}$.

By construction, $S$, $N$ and $W:=W_\delta$ have all the required properties.
\end{proof}

Theorem \ref{thm:main} is an immediate consequence of Theorem \ref{thm:equiv} proved in the preceding subsection and of Theorem \ref{thm:betlee} we prove below.

\begin{proof}[Proof of Theorem \ref{thm:betlee}]
Let $M$, $S \subset \R^{m+1}$, $N$, $W$, $\delta$ and $k$ be as in the statement of Lemma \ref{lem:wallace-type}. For every positive integer $h$, define the nonsingular real algebraic subset $E_h$ of $\R^{m+1}$ and the Nash function $r_h:B_m(1/2) \lra \R$ by setting
\begin{align*}
&E_h:=\{(x,x_{m+1}) \in \R^{m+1} \, | \, (x_{m+1}-1)^2+\|x\|^{2h}-1=0\},\\
&r_h(x):=1-(1-\|x\|^{2h})^{1/2},
\end{align*}
and denote by $\GG(r_h) \subset \R^{m+1}$ the graph of $r_h$. Since it holds
\begin{align*}
&E_h \subset \bar{B}_m(1) \times [0,2] \subset \bar{B}_m(2) \times [-2,2],\\
&E_h \cap (B_m(1/2) \times (-1/2,1/2))=\GG(r_h) \cap (B_m(1/2) \times (-1/2,1/2)),
\end{align*}
points \eqref{eq:s} and \eqref{eq:zcl} imply that
\[
\GG(r_h) \cap S=E_h \cap S=E_h \cap \mr{Zcl}_{\R^{m+1}}(S)
\]
and hence $\GG(r_h) \cap S$ is Zariski closed in $\R^{m+1}$.

Choose an integer $h_0>k$ such that $4^{-h_0+k}<\delta$. For every $h \geq h_0$ and for every $x \in \bar{B}_m(1/2) \setminus \{0\}$, we have:
\begin{align*}
r_h(x) &=\|x\|^{2h}/(1+(1-\|x\|^{2h})^{1/2}) \leq \|x\|^{2h}=\|x\|^{2k}\|x\|^{2h-2k} \leq \\
& \leq \|x\|^{2k}(1/2)^{2h-2k} \leq \|x\|^{2k}4^{-h_0+k} < \delta \|x\|^{2k}.
\end{align*}
Bearing in mind that $0$ is an isolated point of $S \cap \overline{\Omega_{\delta,k}}$ and $W$ is equal to $(S \cap \overline{\Omega_{\delta,k}}) \setminus \{0\}$, we infer at once that $0$ is an isolated point of $\GG(r_h) \cap S$ and $(\GG(r_h) \cap S) \setminus \{0\}=\GG(r_h) \cap W$ for every $h \geq h_0$. It follows that $\GG(r_h) \cap W$ is biregularly isomorphic to a real algebraic set for every $h \geq h_0$.

Fix $h \geq h_0$. Consider the Nash functions $f,g_h:B_m(1/2) \times \R \lra \R$ defined by setting
\[
\text{$f(x,x_{m+1}):=x_{m+1} \;$
and
$\; g_h(x,x_{m+1}):=x_{m+1}-r_h(x)$.}
\]

If $h$ is sufficiently large, $g_h$ is arbitrarily $\mscr{C}^1$-close to $f$ locally at $W$. In this way, by Lemma \ref{lem:thom}, we infer that $N=f^{-1}(0) \cap W$ (and hence $M$) is semi\-algebraically homeomorphic to the real algebraic set $\GG(r_h) \cap W=(g_h)^{-1}(0) \cap W$.

This completes the proof.
\end{proof}


\subsection{Proof of Theorem \ref{thm:a}}

In order to prove point $(\mr{i})$, we need a new version of an important algebraic approximation result, the so-called workhorse theorem of Akbulut and King: see Lemma 2.4 of \cite{akis} or Theorem 2.8.3 of \cite[p.~63]{AK}.

\begin{Lem} \label{lem:approx}
Let $X \subset \R^n$ be a compact real algebraic set, let $Y \subset \R^k$ be a nonsingular real algebraic set, let $H:X \times S^1 \lra Y$ be a Nash map and let $b \in S^1$ such that the restriction of $H$ to $X \times \{b\}$ is a regular map. Choose a point $a$ in $S^1 \setminus \{b\}$ and a compact semialgebraic neighborhood $I_a$ of $a$ in $S^1 \setminus \{b\}$. Then there exist an open semialgebraic neighborhood $U$ of $X$ in $\R^n$, a Nash extension $\widetilde{H}:U \times S^1 \lra Y$ of $H$, a compact semialgebraic neighborhood $U^*$ of $X$ in $U$, a nonsingular real algebraic set $L \subset \R^n \times S^1 \times \R^k$, an open semialgebraic subset $L_0$ of $L$
and a regular map $R:L \lra Y$ with the following properties:
\begin{itemize}
 \item[$(\mr{i})$] If $\pi:\R^n \times S^1 \times \R^k \lra \R^n \times S^1$ denotes the natural projection, then $\pi(L_0)=U \times S^1$ and the restriction $\varpi:L_0 \lra U \times S^1$ of $\pi$ from $L_0$ to $U \times S^1$ is a Nash isomorphism.
 \item[$(\mr{ii})$] $X \times \{b\} \times \{0\} \subset L_0$.
 \item[$(\mr{iii})$] $R \circ i_{L_0} \circ \varpi^{-1}$ is arbitrarily $\mscr{C}^1$-close to $\widetilde{H}$ on $U^* \times I_a$, and $R \circ i_{L_0} \circ \varpi^{-1}=\widetilde{H}$ on $X \times \{b\}$, where $i_{L_0}:L_0 \hookrightarrow L$ denotes the inclusion map.
\end{itemize}
\end{Lem}
\begin{proof}
We organize the proof into two steps.

\textit{Step I.} Let $i_Y:Y \hookrightarrow \R^k$ be the inclusion map and let $q:X \times \{b\} \lra \R^n$ be the regular map sending $(x,b)$ into $H(x,b)$ for every $x \in X$. Choose a Nash extension $\widehat{H}:\R^n \times S^1 \lra \R^k$ of $i_Y \circ H$ and a regular extension $Q:\R^n \times S^1 \lra \R^k$ of $q$. The existence of $\widehat{H}$ is ensured by Efroymson's extension theorem (see \cite[Theorem 8.9.12]{BCR}), the existence of $Q$ by Proposition 3.2.3 of \cite{BCR}.

Let $\varrho:T \lra Y$ be a Nash tubular neighborhood of $Y$ in $\R^k$ with $\varrho$ equal to the closest point map (see \cite[Corollary 8.9.5]{BCR}) and let $U$ be an open semialgebraic neighborhood of $X$ in $\R^n$ with compact closure $\overline{U}$ in $\R^n$ such that $\widehat{H}(\overline{U} \times S^1) \subset T$. Define the Nash map $H^*:\overline{U} \times S^1 \lra \R^k$ by setting $H^*(x,p):=\varrho(\widehat{H}(x,p))$. We remark that the image of $H^*$ is contained in $Y$.

Define $\delta,\Delta \in \R$ as follows:
\begin{align*}
&\delta:=\mr{dist}\big(H^*(\overline{U} \times S^1),\R^k \setminus T\big)>0,\\
&
\textstyle
\Delta:=\max\big\{\sup_{\overline{U} \times S^1}\|H^*\|_k,\sup_{\overline{U} \times S^1}\|Q\|_k\big\},
\end{align*}
where $\|\cdot\|_k$ denotes the usual euclidean norm of $\R^k$. Since $H^*(x,b)=q(x,b)=Q(x,b)$ for every $x \in X$, by restricting $U$ around $X$ if necessary, we can find a compact semialgebraic neighborhood $I_b$ of $b$ in $S^1$ so small that $I_a \cap I_b=\emptyset$ and
\begin{equation} \label{eq:H*Q}
\textstyle
\sup_{\overline{U} \times I_b}\|H^*-Q\|_k<\delta/5.
\end{equation}

Consider a $\mscr{C}^1$-function $\xi:S^1 \lra \R$ such that $\xi=1$ on $S^1 \setminus I_b$, $\xi(b)=0$ and $|\xi| \leq 1$ on the whole $S^1$. By the relative Weierstrass approximation theorem (see \cite[Lemma 2.8.1]{AK}), there exists a regular function $\theta:S^1 \lra \R$ arbitrarily $\mscr{C}^1$-close to $\xi$ such that $\theta(a)=1$ and $\theta(b)=0$. Furthermore, we can assume that
\begin{equation} \label{eq:2}
\text{$|\theta|<2$ on $S^1$}
\end{equation}
and
\begin{equation} \label{eq:Delta}
\textstyle
\text{$\big(\sup_{S^1 \setminus I_b}|1-\theta|\big) \cdot \Delta<\delta/5$}.
\end{equation}

By using again the Weierstrass approximation theorem, we find a regular map $G:\R^n \times S^1 \lra \R^k$ arbitrarily $\mscr{C}^1$-close to $H^*$ on $\overline{U} \times S^1$. In particular, we can suppose that
\begin{equation} \label{eq:GH*}
\textstyle
\sup_{\overline{U} \times S^1}\|G-H^*\|_k<\delta/5.
\end{equation}
Define the regular map $g:\R^n \times S^1 \lra \R^k$ as follows:
\[
g:=\theta \cdot G+(1-\theta) \cdot Q.
\]
Observe that
\begin{equation} \label{eq:gH*}
g-H^*=\theta \cdot (G-H^*)+(1-\theta) \cdot (Q-H^*).
\end{equation}
By construction, $G$ is arbitrarily $\mscr{C}^1$-close to $H^*$ on $\overline{U} \times S^1$ and $\theta$ is arbitrarily $\mscr{C}^1$-close to the function constantly equal to $1$ on $I_a$. Thanks to \eqref{eq:gH*}, it follows that
\begin{equation} \label{eq:g}
\text{$g$ is arbitrarily $\mscr{C}^1$-close to $H^*$ on $\overline{U} \times I_a$.}
\end{equation}
Since $\theta(b)=0$, we have that $g(x,b)=q(x,b) \in Y \;$ for every $x \in X$ and hence
\begin{equation} \label{eq:b}
g(X \times \{b\}) \subset Y.
\end{equation}

Let us show that
\begin{equation} \label{eq:tubular}
\text{$g(\overline{U} \times S^1) \subset T$.}
\end{equation}
To do this, it suffices to prove that $\|g-H^*\|_k<\delta$ on $\overline{U} \times S^1$.

Consider a point $x'=(x,p)$ in $\overline{U} \times I_b$. By \eqref{eq:gH*}, \eqref{eq:H*Q}, \eqref{eq:2} and \eqref{eq:GH*}, we have:
\[
\|g(x')-H^*(x')\|_k <|\theta(p)| \cdot \delta/5+(1+|\theta(p)|) \cdot \delta/5<2\delta/5+3 \delta/5=\delta.
\]
Suppose now that $x'=(x,p) \in \overline{U} \times (S^1 \setminus I_b)$. By \eqref{eq:gH*}, \eqref{eq:2}, \eqref{eq:Delta} and \eqref{eq:GH*}, we have:
\[
\|g(x')-H^*(x')\|_k <|\theta(p)| \cdot \delta/5+|1-\theta(p)| \cdot (2\Delta)<2\delta/5+2 \delta/5<\delta.
\]
This proves \eqref{eq:tubular}. We conclude this step by choosing a compact semialgebraic neighborhood $U^*$ of $X$ in $U$.

\textit{Step II.} The remainder of the proof is quite standard: it follows the classical proof of the workhorse theorem. We will omit the almost identical details.

Let $\mc{Y}:=\{(y,v) \in Y \times \R^k \, | \, v \in T_y(Y)^\perp\}$ be the embedded normal bundle of $Y$ in $\R^k$, let $\eta:\R^n \times S^1 \times \R^k \lra \R^k \times \R^k$ be the regular map given by setting $\eta(x,p,v):=(g(x,p)+v,v)$, let $L:=\eta^{-1}(\mc{Y})$ and let $R:L \lra Y$ be the regular map sending $(x,p,v)$ into $g(x,p)+v$. Since $\eta$ is transverse to $\mc{Y}$ in $\R^{2k}=\R^k \times \R^k$, we have that $L$ is a nonsingular real algebraic subset of $\R^n \times S^1 \times \R^k$.

Denote by $V$ the open semialgebraic subset $g^{-1}(T)$ of $\R^n \times S^1$. Define the Nash maps $\alpha:V \lra Y$ and $\beta:V \lra \R^k$, and the open semialgebraic subset $\widehat{L}_0$ of $L$ by setting
\begin{align*}
&\alpha(x,p):=\varrho(g(x,p)),\\
&\beta(x,p):=\alpha(x,p)-g(x,p),\\
&\widehat{L}_0:=L \cap (V \times \R^k)=\{(x,p,v) \in L \, | \, g(x,p) \in T\}.
\end{align*}

By construction, we see at once that $\widehat{L}_0$ is the graph of $\beta$. By \eqref{eq:tubular}, we infer that $\overline{U} \times S^1 \subset V$. In this way, if we define $L_0:=\widehat{L}_0 \cap (U \times S^1 \times \R^k)$, point $(\mr{i})$ is proved. Point $(\mr{ii})$ follows immediately from \eqref{eq:b} and point $(\mr{iii})$ from \eqref{eq:g} and \eqref{eq:b}, provided we define the Nash map $\widetilde{H}$ as the restriction of $H^*$ from $U \times S^1$ to $Y$.
\end{proof}

\begin{Rem} \label{rem:nonstandard}
The reader observes that Lemma \ref{lem:approx} is a quite nonstandard version of the workhorse theorem. In fact, $R \circ i_{L_0} \circ \varpi^{-1}$ approximates $\widetilde{H}$ on $U^* \times I_a$, but not on the whole $U^* \times S^1$ as in the classical version. Moreover, the maps $R \circ i_{L_0} \circ \varpi^{-1}$ and $\widetilde{H}$ coincide on $X \times \{b\}$ without the assumption that $X$ is nice.  
\end{Rem}

\begin{proof}[Proof of Theorem \ref{thm:a}: point $(\mr{i})$]
We organize the proof into two steps.

\textit{Step I.} Let $X \subset \R^n$ be a compact real algebraic set, let $Y \subset \R^k$ be a nonsingular real algebraic set, let $Z$ be a real algebraic subset of $Y$ and let $f:X \lra Y$ be a Nash map transverse to $Z$ and $\mscr{C}^0$-homotopic to a regular map $g$. By doubling such a homotopy, we obtain a continuous map $h:X \times S^1 \lra Y$ such that $h(x,a)=f(x)$ and $h(x,b)=g(x)$ for every $x \in X$, where $a:=(1,0)$ and $b:=(-1,0)$ are points of $S^1 \subset \R^2$. Since the restriction of $h$ to $X \times \{a,b\}$ is Nash, there exists a Nash map $H:X \times S^1 \lra Y$ (arbitrarily $\mscr{C}^0$-close to $h$ and) equal to $h$ on $X \times \{a,b\}$. The existence of such a Nash map $H$ can be proved by means of an argument similar to the one used in Step I of the proof of Lemma \ref{lem:approx}. For the sake of completeness, we sketch such an argument below.

Let $i_Y:Y \hookrightarrow \R^k$ be the inclusion map, let $\varrho:T \lra Y$ be a Nash tubular neighborhood of $Y$ in $\R^k$ and let $I$ be a small open semialgebraic neighborhood of $\{a,b\}$ in $S^1$. By Efroymson's extension theorem, there exists a Nash map $h^*:X \times S^1 \lra \R^k$, which coincides with $i_Y \circ h$ on $X \times \{a,b\}$. If $I$ is sufficiently small around $\{a,b\}$ in $S^1$, then $h^*$ is arbitrarily $\mscr{C}^0$-close to $i_Y \circ h$ on $X \times I$. By the Weierstrass approximation theorem, we can find a polynomial map $h^{**}:X \times S^1 \lra \R^k$ arbitrarily $\mscr{C}^0$-close to $i_Y \circ h$ on the whole $X \times S^1$. Choose a Nash function $\xi:S^1 \lra \R$ such that $\xi(a)=\xi(b)=0$, $|\xi|<2$ on $S^1$ and $\xi$ is arbitrarily $\mscr{C}^0$-close to $1$ on $S^1 \setminus I$. It follows that the Nash map $H^*:X \times S^1 \lra \R^k$, given by $H^*:=\xi \cdot  h^{**}+(1-\xi) \cdot h^*$, is arbitrarily $\mscr{C}^0$-close to $i_Y \circ h$ on $X \times S^1$ and it is equal to $i_Y \circ h$ on $X \times \{a,b\}$. In particular, the image of $H^*$ is contained in $T$. Now, it suffices to set $H(x,t):=\varrho(H^*(x,p))$.

\textit{Step II.} Let us apply Lemma \ref{lem:approx} to $H$. Let $I_a$, $U$, $U^*$, $\widetilde{H}:U \times S^1 \lra Y$, $L$, $L_0$, 
$R:L \lra Y$ and $\pi:\R^n \times S^1 \times \R^k \lra \R^n \times S^1$ be as in the statement of such a lemma. Denote by $\varpi:L_0 \lra U \times S^1$ the Nash isomorphism obtained by restricting $\pi$. Define the Nash maps $r:U \times S^1 \lra Y$ and $r_a,\widetilde{H}_a:U \lra Y$ by setting $r(x,p):=R(\varpi^{-1}(x,p))$, $r_a(x):=r(x,a)$ and $\widetilde{H}_a(x):=\widetilde{H}(x,a)$. Observe that $\widetilde{H}_a$ is a Nash extension of $f$ on $U$. By point $(\mr{ii})$ of Lemma \ref{lem:approx}, we have that $r(x,b)=g(x)$ for every $x \in X$.

Equip $\R^{n+2}$ with the coordinates $(x_1,\ldots,x_{n+2})$ and indicate by $J$ the coordinate hyperplane $x_{n+2}=0$ of $\R^{n+2}$. Consider the compact Nash set $S_*:=r^{-1}(Z) \cap (X \times S^1)$. We have that $J \cap S_*$ is the disjoint union of $(r_a)^{-1}(Z) \times \{a\}$ and $g^{-1}(Z) \times \{b\}$. By using point $(\mr{iii})$ of the same lemma, we know that $r$ is arbitrarily $\mscr{C}^1$-close to $\widetilde{H}$ on $U^* \times I_a$. In particular, $r_a$ is arbitrarily $\mscr{C}^1$-close to $\widetilde{H}_a$ on $U^*$. Since $f$ is transverse to $Z$, it follows that $J$ is transverse to $S_*$ in $\R^{n+2}$ locally at $(r_a)^{-1}(Z) \times \{a\}$. Furthermore, Corollary \ref{cor:thom} implies that $(r_a)^{-1}(Z)$ is semialgebraically homeomorphic to $f^{-1}(Z)$.

Define the real algebraic set $S':=R^{-1}(Z)$ and the compact Nash set $S:=L_0 \cap S'$. Since $S$ is equal to the compact Nash set $\varpi^{-1}(S_*)$, we infer that $S$ is a union of connected components of $S'$ (and hence of $\mr{Zcl}_{\R^{n+2+k}}(S)$), and $(J \times \R^k) \cap S$ is the disjoint union of the real algebraic set $g^{-1}(Z) \times \{b\} \times \{0\}$ and of the compact Nash set $N:=\varpi^{-1}((r_a)^{-1}(Z) \times \{a\})$, which is semialgebraically homeomorphic to $f^{-1}(Z)$. Finally, $J \times \R^k$ is transverse to $S$ in $\R^{n+2+k}$ locally at $N$. We have just proved that $f^{-1}(Z)$ is strongly asym-Nash cobordant to a real algebraic set. Theorem \ref{thm:main} completes the proof.
\end{proof}

\begin{Rem} \label{rem:2.6i}
In the preceding proof, we cannot conclude that $J \times \R^k$ is tran\-sverse to $S$ in $\R^{n+2+k}$ locally at $g^{-1}(Z) \times \{b\} \times \{0\}$. This is due to the fact that we do not require that $g$ is transverse to $Z$. In this sense, we can say that the asymmetry of Definition \ref{def:nash-cobordant} is necessary.
\end{Rem}

\begin{Rem} \label{rem:2.6i-bis}
If $f$ is null homotopic (and hence $\mscr{C}^0$-homotopic to a regular map) and $\dim(Z)<\dim(Y)$, then the argument used in the preceding proof ensures that $f^{-1}(Z)$ is a Nash boundary.
\end{Rem}

Let us prove point $(\mr{ii})$.

\begin{proof}[Proof of Theorem \ref{thm:a}: point $(\mr{ii})$]
As in the preceding proof, we will show that $f^{-1}(Z)$ is strongly asym-Nash cobordant to a real algebraic set. By Theorem \ref{thm:main}, the proof will be complete.

Let $F:V \lra Y$ be a Nash extension of $f$. Consider the Nash map $F \times \mr{id}_V:V \lra Y \times V$, where $\mr{id}_V:V \lra V$ indi\-cates the identity map on $V$. By hypothesis, $Y$ and $V$ have totally algebraic homology. In this way, the K\"unneth formula implies that $Y \times V$ has totally alge\-braic homology as well. Thanks to Lemma 2.7.1 of \cite{AK}, we have that $F \times \mr{id}_V$ is unoriented bordant to a regular map. Equivalently, there exist a compact Nash submanifold $B$ of some $\R^{n+1}$ containing $V$, a nonsingular real algebraic subset $V'$ of $\R^n$ contained in $B \setminus V$ and a Nash map $H:B \lra Y \times V$ such that $\R^n \cap B=V \cup V'$, $\R^n$ is transverse to $B$ in $\R^{n+1}$, the restriction of $H$ to $V$ is equal to $F \times \mr{id}_V$ and the restriction of $H$ to $V'$ is a regular map. Now, by applying Theorem 2.8.3 of \cite[p. 64]{AK} to $H$, we obtain a positive integer $k$, a nonsingular real algebraic subset $L$ of $\R^{n+1+k}=\R^{n+1} \times \R^k$, a union $L_0$ of connected components of $L$ and regular maps $P:L \lra Y$ and $Q:L \lra V$ with the following properties:
\begin{itemize}
 \item[$(a)$] If $\pi:\R^{n+1} \times \R^k \lra \R^{n+1}$ denotes the natural projection, then $\pi(L_0)=B$ and the restriction $\varpi:L_0 \lra B$ of $\pi$ from $L_0$ to $B$ is a Nash isomorphism.
 \item[$(b)$] $V' \times \{0\} \subset L_0$.
 \item[$(c)$] $(P \times Q) \circ i_{L_0} \circ \varpi^{-1}$ is arbitrarily $\mscr{C}^1$-close to $H$ on $B$ and $(P \times Q) \circ i_{L_0} \circ \varpi^{-1}=H$ on $V'$, where $i_{L_0}:L_0 \hookrightarrow L$ denotes the inclusion map.
\end{itemize}

Let $p:B \lra Y$ and $q:B \lra V$ be the Nash maps sending $x$ into $p(x):=(P \circ i_{L_0} \circ \varpi^{-1})(x)$ and $q(x):=(Q \circ i_{L_0} \circ \varpi^{-1})(x)$, respectively. By points $(b)$ and $(c)$, $p \times q$ coincides with $H$ on $V'$. In particular, $p$ and $q$ are regular maps on $V'$. By using point $(c)$ again, we infer that:
\begin{itemize}
 \item[$(c')$] $p$ is arbitrarily $\mscr{C}^1$-close to $F$ on $V$.
 \item[$(c'')$] $q$ is arbitrarily $\mscr{C}^1$-close to $\mr{id}_V$ on $V$. In particular, it is a Nash submersion locally at $V$ in $B$ and its restriction $q':V \lra V$ to $V$ is a Nash automorphism of $V$ arbitrarily $\mscr{C}^1$-close to $\mr{id}_V$.
\end{itemize}

Let $X_*$ and $\widetilde{X}$ be the compact Nash sets and let $X'$ be the real algebraic set defined by setting $X_*:=q^{-1}(X)$, $\widetilde{X}:=X_* \cap V$ and $X':=X_* \cap V'$. Observe that $\R^n \cap X_*$ is equal to the disjoint union of $\widetilde{X}$ and of $X'$. Moreover, by the second part of point $(c'')$, we have that $\R^n$ is transverse to $X_*$ in $\R^{n+1}$ locally at $\widetilde{X}$, and $\widetilde{X}=(q')^{-1}(X)$ is Nash isomorphic to $X$.

Consider the compact Nash set $S_*:=p^{-1}(Z) \cap X_*$. We have that $\R^n \cap S_*$ is the disjoint union of the compact Nash set $\widetilde{X}_*:=p^{-1}(Z) \cap \widetilde{X}$ and of the real algebraic set $X'_*:=p^{-1}(Z) \cap X'$. By combining points $(c')$ and $(c'')$ with the fact that $f$ is transverse to $Z$, we obtain at once that $\R^n$ is transverse to $S_*$ in $\R^{n+1}$ locally at $\widetilde{X}_*$. Furthermore, Corollary \ref{cor:thom} ensures that $\widetilde{X}_*$ is semialgebraically homeomorphic to $(F \circ q)^{-1}(Z) \cap \widetilde{X}=(q')^{-1}(f^{-1}(Z))$ (and hence to $f^{-1}(Z)$).

Define the real algebraic set $S':=P^{-1}(Z) \cap Q^{-1}(X)$ and the compact Nash set $S:=L_0 \cap S'$. Denote by $O$ the coordinate hyperplane $x_{n+1}=0$ of $\R^{n+1+k}=\R^{n+1} \times \R^k$. Since $S$ is equal to $\varpi^{-1}(S_*)$, we infer that $S$ is a union of connected components of $S'$ (and hence of $\mr{Zcl}_{\R^{n+1+k}}(S)$), and $O \cap S$ is the disjoint union of the real algebraic set $X'_* \times \{0\}$ and of the compact Nash set $N:=\varpi^{-1}(\widetilde{X}_*)$, which is semialgebraically homeomorphic to $f^{-1}(Z)$. Finally, $O$ is transverse to $S$ in $\R^{n+1+k}$ locally at $N$. This proves that $f^{-1}(Z)$ is strongly asym-Nash cobordant to a real algebraic set, as desired.
\end{proof}

\begin{Rem} \label{rem:2.6ii}
In the preceding proof of point $(\mr{ii})$, similarly to the one of point $(\mr{i})$, we do not know if $O$ is transverse to $S$ in $\R^{n+1+k}$ locally at $X'_* \times \{0\}$.
\end{Rem}

It is worth noting that the argument used in the latter proof implies a new algebraic approximation result. The result is as follows.

\begin{Thm} \label{thm:approx-new}
Let $V$ and $Y$ be nonsingular real algebraic sets with totally algebraic homology and let $F:V \lra Y$ be a Nash map. Suppose that $V$ is compact. Then there exist a compact nonsingular real algebraic set $\widetilde{V}$, a regular map $\widetilde{F}:\widetilde{V} \lra Y$ and a Nash isomorphism $\Pi:\widetilde{V} \lra V$ such that $\Pi$ is a regular map and $\widetilde{F}$ is arbitrarily $\mscr{C}^{\infty}$-close to $F \circ \Pi$.
\end{Thm}
\begin{proof}
Let $V'$, $B$, $H$, $L$, $L_0$, $P$, $Q$ and $\varpi$ be as in the preceding proof. Observe that the nonsingular Nash hypersurface $\varpi^{-1}(V)$ of $L_0$ is homologous to real algebraic set $V' \times \{0\}$ in $L_0$. In this way, by Theorem 2.8.2 of \cite{AK}, there exist a nonsingular real algebraic hypersurface $\widetilde{V}$ of $L$ contained in $L_0$ and a smooth (or better Nash) automorphism of $L_0$ arbitrarily $\mscr{C}^\infty$-close to $\mr{id}_{L_0}$ sending $\widetilde{V}$ into $\varpi^{-1}(V)$. Let $\varpi':\varpi^{-1}(V) \lra V$ be the restriction of $\varpi$ from $\varpi^{-1}(V)$ to $V$, and let $\widetilde{F}:\widetilde{V} \lra Y$ and $\Pi:\widetilde{V} \lra V$ be the restrictions of $P$ and of $Q$ to $\widetilde{V}$, respectively. Since $P$ is arbitrarily $\mscr{C}^\infty$-close to $F \circ \varpi'$ on $\varpi^{-1}(V)$ and $Q$ is arbitrarily $\mscr{C}^\infty$-close to $\varpi'$ on $\varpi^{-1}(V)$, it follows that $\widetilde{F}$ is arbitrarily $\mscr{C}^\infty$-close to $F \circ \Pi$ as well.
\end{proof}

The reader compares the statement and the proof of the preceding theorem with the statement and the proof of Proposition 2.8.8 of \cite{AK}, in which Akbulut and King prove a similar result for smooth maps homotopic to regular maps between nonsingular real algebraic sets, whose homology is not necessarily totally algebraic.

\begin{Rem} \label{rem:2.6ii-bis}
We point out that, starting from Theorem \ref{thm:approx-new}, one can prove point $(\mr{ii})$ of Theorem \ref{thm:a} by a direct application of Corollary \ref{cor:thom}.
\end{Rem}

\subsection{Proof of Theorem \ref{thm:b}}

\textit{Point $(\mr{i})$}. Let $A$ be a Nash boundary and let $B \subset \R^n$ be a compact Nash set such that $M=A \times B$. Replacing $A$ with one of its semialgebraically homeomorphic copy if necessary, we can assume that there exist a compact Nash subset $S$ of some $\R^{m+1}$ with $\R^m \cap S=A$ and a semialgebraic Whitney stratification $\{U_i\}_i$ of an open semialgebraic neighborhood $U$ of $A$ in $S$ such that $\R^m$ is transverse to each $U_i$ in $\R^{m+1}$.

Choose a semialgebraic Whitney stratification $\{B_j\}_j$ of $B$ (see \cite[Section 9.7]{BCR} for the existence of such a stratification). Define the compact Nash subset $S^*:=S \times B$ of $\R^{m+1+n}=\R^{m+1} \times \R^n$ and the semialgebraic Whitney stratification $\{U_i \times B_j\}_{i,j}$ of the open semialgebraic neighborhood $U \times B$ of $M$ in $S^*$. Denote by $J$ the coordinate hyperplane $x_{m+1}=0$ of $\R^{m+1+n}$. It is immediate to verify that $J \cap S^*=M$ and $J$ is transverse to each stratum $U_i \times B_j$ in $\R^{m+1+n}$. This proves that $M$ is a Nash boundary.

\textit{Point $(\mr{ii})$.} Let $P$ and $Q$ be proper Nash subsets of some $\R^m$ such that $Q$ is compact and $M=P \cap Q$. Since $P$ is a proper subset of $\R^m$, we may suppose that $0 \not\in P$. Assume that $P$ is transverse to $Q$ in $\R^m$; namely, that there exist a semialgebraic Whitney stratification $\{V_k\}_k$ of an open semialgebraic neighborhood $V$ of $M$ in $P$ and a semialgebraic Whitney stratification $\{W_h\}_h$ of an open semialgebraic neighborhood $W$ of $M$ in $Q$ such that each stratum $V_k$ is transverse to each stratum $W_h$ in $\R^m$.

Choose two positive real numbers $r$ and $r'$ in such a way that $r'>r$, $\bar{B}_m(r) \cap P=\emptyset$ and $Q \subset \bar{B}_m(r')$. Such real numbers exist, because $P$ is closed in $\R^m$, $0 \not\in P$ and $Q$ is compact. Equip $\R^{m+1}=\R^m \times \R$ with the coordinates $(x,x_{m+1})=(x_1,\ldots,x_m,x_{m+1})$, and consider the polynomial map $H:\R^{m+1} \lra \R^m$ and the Nash subset $S$ of $\R^{m+1}$ defined by setting
\[
\textstyle
H(x,x_{m+1}):=\left(\frac{r'-r}{r} \, x_{m+1}^2 +1 \right)x
\quad \mbox{and} \quad
S:=(P \times \R) \cap H^{-1}(Q).
\]
By the choice of $r$ and $r'$, it is immediate to verify that $S$ is contained in $\bar{B}_m(r') \times [-1,1]$, and hence it is compact. Moreover, we have that $\R^m \cap S=M$. Observe that $\{V_k \times \R\}_k$ and $\{H^{-1}(W_h)\}_h$ are semialgebraic Whitney stratifications of $V \times \R$ and of $H^{-1}(W)$, respectively. The transversality between the strata $V_k$ and the strata $W_h$ in $\R^m$ implies the transversality between the strata $V_k \times \R$ and the strata $H^{-1}(W_h)$ in $\R^{m+1}$ locally at $(V_k \cap W_h) \times \{0\}$. Furthermore, $\R^m$ is transverse to each intersection $(V_k \times \R) \cap H^{-1}(W_h)$ in $\R^{m+1}$ locally at $(V_k \cap W_h) \times \{0\}$. It follows that, for a sufficiently small positive real number $\varepsilon$, the family $\{(V_k \times (-\varepsilon,\varepsilon)) \cap H^{-1}(W_h)\}_{k,h}$ is a semialgebraic Whitney stratification of the open semialgebraic neighborhood $(\R^m \times (-\varepsilon,\varepsilon)) \cap S$ of $M$ in $S$, and $\R^m$ is transverse to each element of such a stratification in $\R^{m+1}$. This proves that $M$ is a Nash boundary.



\vspace{1.5em}

\begin{flushleft}
\textsc{Riccardo Ghiloni}\\ 
Department of Mathematics,
 University of Trento\\
Via Sommarive 14, 38123 Povo-Trento, Italy\\
e-mail: \texttt{ghiloni@science.unitn.it}\\
\vspace{1em}
\textsc{Alessandro Tancredi}\\
Department of Mathematics and Informatics,
 University of Perugia\\
Via Vanvitelli 1, 06123 Perugia, Italy \\
e-mail: \texttt{alessandro.tancredi@unipg.it}
\end{flushleft}


\end{document}